\documentclass[11pt]{amsart}
\usepackage{amsmath}
\usepackage{amsfonts}
\usepackage{amssymb}
\usepackage{amsthm}
\usepackage{graphicx}
\usepackage{mathrsfs,eucal,xspace}

\newtheorem{theorem}{Theorem}[section] % 1st argument is your name for it
\newtheorem{lemma}[theorem]{Lemma}     % 2nd argument is what is printed
\newtheorem{corollary}[theorem]{Corollary}
\newtheorem{proposition}[theorem]{Proposition}
\newtheorem{definition}[theorem]{Definition}
\newtheorem*{thma}{Theorem}

\newcommand{\const}{\mathrm{const}\,}

\newcommand{\spam}{\mathop{\mathrm{span}}}

\newcommand{\sphere}{\mathbb{S}^d}
\newcommand{\nat}{\mathbb{N}}
\newcommand{\reals}{\mathbb{R}}
\newcommand{\comps}{\mathbb{C}}

\newcommand{\dif}{\mathrm{d}}

\newcommand{\dist}{\mathrm{dist}}
\newcommand{\order }{ 2m}
\newcommand{\miniorder}{{2m}}% 

\newcommand{\val}{2m}

\newcommand{\s}{\vec{s}\;}
\newcommand{\e}{\boldsymbol{e}}
\renewcommand{\L}{\mathcal{L}}
\renewcommand{\k}{\boldsymbol{k}}
\newcommand{\kernel}{\k}

\renewcommand{\a}{\boldsymbol{a}}
\renewcommand{\l}{{\vec{\ell}}\;}

\title{Polyharmonic Approximation on the Sphere} 
\author{T. Hangelbroek}
\address{Thomas Hangelbroek, Department of Mathematics, Texas A\&M University, College Station, TX 77843}
\email{hangelbr@math.tamu.edu}
\thanks{Thomas Hangelbroek is supported by an NSF Postdoctoral Fellowship}
\begin{document}
\maketitle
\begin{abstract}
The purpose of this article is to provide new error estimates for a popular type of SBF approximation on the sphere: approximating by linear combinations of Green's functions of polyharmonic differential operators. We show that the $L_p$ approximation order for this kind of approximation is $\sigma$ for functions having $L_p$ smoothness $\sigma$ (for $\sigma$ up to the order of the underlying differential operator, just as in univariate spline theory). This improves previous error estimates, which penalized the approximation order when measuring error in $L_p$, $p>2$ and held only in a restrictive setting when measuring error in $L_p$,  $p<2$.
\end{abstract}

\section{Introduction} 
\label{intro}
Spherical basis functions (or SBFs) have been used with much success in multivariate approximation theory, statistics and a multitude of other scientific disciplines. 
At the heart of the SBF methodology is the creation of an approximant 
$$s_{\Xi}(x) = \sum_{\xi \in \Xi} A_{\xi} \k(x\cdot \xi)$$
by taking a linear combination of rotations of
a fixed kernel $(x,\alpha)\mapsto \k(x\cdot \alpha)$ (known as an SBF, or, sometimes, a {\em zonal} kernel). 

The success of the SBF methodology derives from its ability to generate approximants from data having arbitrary geometry -- a desirable quality on spheres, where geometry of data is always essentially unstructured: for an arbitrary spacing, there are no regular distributions of points on the sphere, meaning that approximation techniques requiring grids, regular triangulations, or other geometrical props do not work in this setting. 
Interpolation  \cite{Gol}%(on the unit circle)
, \cite{Free}%(on the sphere in $\reals^3$)
, \cite{JSW} and \cite{HuMo2} and
other SBF approximation methods
\cite{MNW2}, \cite{MNiYao}, \cite{Wahba}, \cite{LNWW} 
(see bibliography in \cite{Free} for even more examples), 
are both frequently used to fit scattered data on the sphere.

Our focus is not on how to treat spherical data, but how to approximate smooth %target 
functions using SBF approximants, having access to as much information about the target function as necessary.
The choice of coefficients $(A_{\xi})_{\xi\in\Xi}$ is a crucial element in the performance of the approximation, but,
at the outset, we are not focused on a specific method of choosing coefficients.
Instead, we wish to investigate the approximation power of this methodology for a robust family -- the polyharmonic kernels (see Definition \ref{ker_def}) -- 
rather than any specific implementation or algorithm; 
the main concern is to establish accurate error analysis for 
approximation from spaces of polyharmonic SBFs,
$S(\k,\Xi):= 
\spam_{\xi \in \Xi} \k(\cdot,\xi)+ \Pi,$
where a low dimensional space of elementary functions, $\Pi$, may be added 
to the span of the SBF.

The method for gauging the approximation power is the $L_p$ {\em approximation order}, which measures the decay of the error in approximating from $S(\k,\Xi)$ as $\Xi$ becomes dense in $\sphere$.  
For target functions $f$ from a class $\mathcal{F}$, the approximation order is the largest exponent $s$ so that  
\[\|f-s_{f,\Xi}\|_{L_p(\mathbb{S}^d)} = \mathcal{O}(h^s)\]
where $h$, the `fill distance', measures the density of $\Xi$ in $\sphere$ 
(see the following section for a precise definition of fill distance). 
In this setting, the rate is given in terms of the density of the centers $\Xi$, and depends strongly on the class $\mathcal{F}$ of target functions. 

When approximants are chosen from a predetermined linear space, independent of the target function, as is the case here (in contrast to {\em nonlinear approximation}, where the set of centers $\Xi$ could be chosen independent of $f$), precise approximation theory ties the $L_p$ approximation order to the $L_p$ smoothness of the target function, e.g., by measuring the error in terms of an $L_p$ modulus of smoothness or by
selecting target functions in an $L_p$ Sobolev or Besov space, $\mathcal{F} =W_p^s$ or $B_{p,q}^s$.  

The prevailing method for estimating error for SBF (and, more generally, kernel) approximation has been to assume the target function resides in a reproducing kernel Hilbert space, often called the {\em native space}, for which the SBF acts as the reproducing kernel. 
Quite often, the native space is actually an $L_2$ Sobolev space. 
One drawback of this  approach  has been that it
precludes finding faster rates for functions with more smoothness, or slower rates for less smoothness. Another drawback is that the $L_p$ approximation orders degrade as $p$ aberrates from $2$
-- see, e.g., \cite{Hang1} for an example of this criticism for `radial basis functions' (or RBFs) in domains in $\reals^2$.  
We remark that \cite[Corollary 3.5]{LNWW}, and \cite[Corollary 3 (a)]{JSW} are examples of this phenomenon, but we place special emphasis on the results of Hubbert and Morton \cite[Theorem 3.4, 3.8]{HuMo1},  because their results are the current state-of-the-art for the setting of this article. We paraphrase their result.
%%%%%%%%%%%%%%%%%%%%%%%%%%%%%%%%%%%%%%%%%%%%%%%%%%%%%%%%%%%%%%%5
%Hubbert and Morton Theorem
%%%%%%%%%%%%%%%%%%%%%%%%%%%%%%%%%%%%%%%%%%%%%%%%%%%%%%%%%%%%%%%
\begin{thma}[Hubbert, Morton]
 For an SBF $\k$ having native space $W_2^{m}(\sphere)$,  and for sufficiently dense centers $\Xi$,
if $f\in W_2^{m}(\sphere)$ then the SBF interpolant $s_f$ satisfies:
%%%%%%%%%%%%%
\begin{equation*}
\|f-s_{f}\|_{L_p(\sphere)}  
= 
\mathcal{O}
\left(
  h^{m-(\frac{d}{2} - \frac{d}{p})_{+}}
\right)
\end{equation*} 
%%%%%%%%%%%%%
If $f\in W_2^{2m}(\sphere)$, $s_f$ satisfies:
%%%%%%%%%%%%%
\begin{equation*}
\|f-s_{f}\|_{L_p(\sphere)}  
= 
\mathcal{O}
\left(
  h^{2m -(\frac{d}{2} - \frac{d}{p})_{+}}
\right)
\end{equation*}
%%%%%%%%%%%
\end{thma}
%%%%%%%%%%%%%%%%%%%%%%%%%%%%%%%%%%%%%%%%%%%%%%%%%%%%%%%%%%%%%%%%%%%%%%%%%%%%%
When $p>2$, each approximation order is penalized by subtracting a positive term: $\frac{d}{2} - \frac{d}{p}$; when $p<2$, the space of target functions is an $L_2$ Sobolev space of the form $W_2^{\sigma}(\sphere)$ which is embedded in $W_p^{\sigma}(\sphere)$. 
This should be contrasted with $M^{\text{th}}$ order univariate spline approximation, which provides $L_p$ approximation order $\sigma$ for functions in  $W_p^{\sigma}(\sphere)$ for a range of $0< \sigma \le M)$, where $M$ is the `saturation' order -- the rate beyond which any increase in smoothness fails to produce an increased rate of convergence. 
Our main results show (in Theorem \ref{Main} and its corollaries) for a polyharmonic kernel, $\k$, satisfying the conditions of the above theorem, and for a target function $f$ having smoothness $\sigma\le 2m$ in $L_p$ there is $s_{f,\Xi}\in S(\k,\Xi)$ so that
$\|f - s_{f,\Xi}\|_p = \mathcal{O}(h^{\sigma})$.

In this paper, we develop an approximation scheme delivering novel error estimates for a robust family of SBFs: the `polyharmonic' kernels. 
This is the family of Green's functions of iterated and perturbed Laplace \!--\! Beltrami operators (see Definition
\ref{ker_def} for a precise definition). 
Such kernels have been studied by Freeden and his collaborators, cf.~\cite{Free} and references therein. They include the Green's functions for $\Delta ^m$, and, thus, are direct generalizations of the periodic ``Bernoulli splines'' (famously studied in \cite{Gol})  and are, in some sense, the spherical analogues of the ``surface splines'' used in $\reals^d$. On the other hand, the SBFs obtained by directly restricting the
$\reals^{d+1}$ surface splines to $\sphere$ 
are, perhaps surprisingly, often represented in this family.

The scheme developed in this article is based on replacing the kernel in an integral identity by a linear combination of (few) scattered rotations of the kernel. This method has recently been introduced by DeVore and Ron in \cite{DeRo} where it was used to obtain nonlinear and local results for RBF approximation in the boundary-free, Euclidean setting. Later, it was used in \cite{Hang1}, \cite{Hang2} to provide precise approximation orders for RBF approximation in domains in $\reals^2$. 

The layout of this article is as follows. In Section 2 we discuss some basics of analysis on spheres. Section 3 introduces the kernels used in this paper and shows that they can be expressed as a sum of surface splines. In Section 4 we establish a basic strategy for exchanging the kernel by a linear combination of its copies. Section 5 estimates the error in making this exchange, while Section 6 collects our main results.

%%%%%%%%%%%%%%%%%%%%%%%%%%%%%%%%%%%%%%%%%%%%%%%%%%%%%%5
\section{Background}\label{Sec:Background}
We denote by $\sphere$ the unit sphere in $\reals^{d+1}$, and by $\omega_d$ we denote its volume. The distance between two points, $x$ and $\alpha$, on the sphere is written  $\dist(x,\alpha) := \arccos (x\cdot \alpha)$. The basic neighborhood is the spherical `cap' $C(\alpha,\rho):= \{x\in\sphere: \dist(x,\alpha)<\rho\}$. Throughout this article, $\Xi$ is assumed to be a finite subset of $\sphere$, and the `fill distance', 
$$h := h(\Xi,\sphere):= \max_{\alpha\in \sphere}\dist(\alpha,\Xi),$$ 
measures the density of $\Xi$ in $\sphere$

The spherical harmonic, as studied in \cite{Mull}, is the basic tool of Fourier analysis on the sphere. For each eigenvalue, $\nu_{\ell} := \ell(\ell+d-1)$ of the Laplace \!--\! Beltrami operator $\Delta$ on  $\mathbb{S}^d$, 
there corresponds an eigenspace of `spherical harmonics' of exact degree $\ell$, called $\mathcal{H}_{\ell}$, having dimension $N(d,\ell) := \frac{(2\ell +d-1)\Gamma(\ell +d-1)}{\Gamma(\ell+1)\Gamma(d)}$ with orthonormal (in the sense of $L_2$) basis $(Y_{m,\ell})_{m=1}^{N(d,\ell)}.$  The space of spherical harmonics of degree less than or equal to $L$ is denoted $\Pi_{L} = \sum_{\ell \le L}\mathcal{H}_{\ell}$. 

In this article, our focus is on {\em zonal kernels}. 
These are kernels on the sphere having the form 
$(x,\alpha)\mapsto \phi(x\cdot \alpha)$, 
where $\phi:[-1,1] \to \reals$. 
Such kernels, being the composition of an inner product with a univariate function, can be expressed in terms of an expansion in orthogonal polynomials. The Gegenbauer (or ultraspherical) polynomials,  $(P_{\ell}^{(\lambda)})_{\ell=0}^{\infty}$, are orthogonal on $[-1,1]$ with respect to the weight $(1-t^2)^{\lambda-1/2}$.
We expand zonal functions on $\sphere$ using 
$(P_{\ell}^{(\lambda_d)})_{\ell=0}^{\infty}$, with $\lambda_d:= \frac{d-1}{2}$.
Gegenbauer coefficients are 
$$a_{\ell} := \int_{-1}^1 \phi(t)P_{\ell}^{(\lambda_d)}(t)(1-t^2)^{(d-2)/2}\dif t$$
and the expansion is 
$\phi(x\cdot\alpha) = \sum_{\ell = 0}^{\infty} a_{\ell} P_{\ell}^{(\lambda_d)}(x\cdot\alpha)$. 
This can be expressed, via the addition theorem for spherical harmonics \cite[Theorem 2]{Mull}, as:
%%%%%%%%%%%%%%
%Kernels & Addition theorem
%%%%%%%%%%%%%
\[\phi(x\cdot \alpha) 
= 
\sum_{\ell = 0}^{\infty} 
  \frac{\lambda_d+\ell}{\omega_d\lambda_d}\, 
  \widehat{\phi}(\ell) 
  P_{\ell}^{(\lambda_d)}(x\cdot \alpha)
= 
\sum_{\ell = 0}^{\infty}
\sum_{m=1}^{N(d,\ell)} 
  \widehat{\phi}(\ell) Y_{\ell,m}(x)Y_{\ell,m}(\alpha)
\]
%%%%%%%%%%%%%
%%%%%%%%%%%%%
where we supplant the Gegenbauer coefficient $a_{\ell}$ by the Fourier coefficient 
$\widehat{\phi}(\ell):= \frac{\omega_d\lambda_d}{\lambda_d+\ell}a_{\ell}$.
We note that the polynomials used by M{\"u}ller, \cite{Mull}, which he calls Legendre polynomials and denotes by $\mathcal{P}_{\ell}$ (suppressing the dependence on $d$),  are normalized in $L_{\infty}$:
they satisfy 
$\|\mathcal{P}_{\ell}\|_{L_\infty[-1,1]} = \mathcal{P}_{\ell}(1) =1$. 
The Gegenbauer polynomials used here are 
normalized in 
$L_2([-1,1];(1-t^2)^{\lambda-1/2})$, 
and are related to M{\"u}ller's Legendre polynomials by:
$P_{\ell}^{(\lambda_d)}= {{\ell +2\lambda_d - 1}\choose{\ell}} P_{\ell}$. 
Basics of Gegenbauer polynomials can be found in \cite[Section 4.7]{Szeg}. 
A key result relates the smoothness of the kernel $\phi$ with the decay of its Fourier coefficients, $\widehat{\phi}(\ell)$.
%%%%%%%%%%%%%%%%%%%%%%%%%%%5
%:smoothness Legendre
%%%%%%%%%%%%%%%%%%%%%%%%5%%
\begin{proposition}\label{smoothness}
If 
$\sum_{\ell = 0}^{\infty} |\hat{\phi}(\ell)| \ell^{d + 2k-1} <\infty$ 
then $\phi\in C^{k}[-1,1]$.
\end{proposition}
% %%%%%%%%%%%%%%%%%%%%%%55
% %
% %%%%%%%%%%%%%%%%%%%%%%%%%
\begin{proof} 
From \cite[Equation (4.7.14)]{Szeg}, 
observe that the derivative of a Gegenbauer polynomial satisfies
$
\frac{d}{dt} P_{\ell+1}^{(\lambda_d)}(t) 
= 
2\lambda_d P_{\ell}^{(\lambda_{d+2})}(t).
$ 
Hence, for $k\le \ell$,
\[
\frac{d^k}{dt^k} P_{\ell}^{(\lambda_d)}(t) 
= 
2^k\lambda_d\lambda_{d+2}\cdots \lambda_{d+2k-2} 
P_{\ell-k}^{(\lambda_{d+2k})}(t), 
\]
while for $\ell<k$, the polynomial $P_{\ell}^{(\lambda_d)}$ is of degree at most $k-1$ and is  annihilated by $\frac{d^k}{dt^k}$.)
Since $\omega_{d+2}\lambda_{d+2}= \pi \omega_d$, it follows that 
$$
\left(
  \frac{\lambda_d + \ell }{\lambda_d \; \omega_d}
\right) 
\frac{d^k}{dt^k} P_{\ell}^{(\lambda_d)}(t) 
= 
2 (2\pi)^{k-1}  
\left(\frac{\lambda_d+\ell}{\omega_{d+2k-2}}\right) 
P_{\ell-k}^{(\lambda_{d+2k})}(t).
$$ 
Utilizing a uniform bound on Gegenbauer polynomials (\cite[Theorem 7.33.1]{Szeg}), 
$\max_{-1\le t\le 1}|P_{\ell}^{(\lambda)}(t)| = {{\ell +2\lambda - 1}\choose{\ell}}$, 
when $\lambda\ge 0$ we see that
\begin{eqnarray*}
\left|
  \left(
    \frac{\lambda_d + \ell}{\lambda_d \omega_d}
  \right) 
  \frac{d^k}{dt^k} 
  P_{\ell}^{(\lambda_d)}(t)
\right|
&\le&  
2 (2\pi)^{k-1}  
\left(
  \frac{\lambda_d+\ell}{\omega_{d+2k-2}}
\right) 
{{\ell +d+k - 2}\choose{\ell-k}}\\
& \le& 
C_{d,k} \, \ell^{d+2k-1}.
\end{eqnarray*}
The result follows because  the series 
$ 
\sum_{\ell=0}^{\infty} 
  \frac{\lambda_d+\ell}{\omega_d\lambda_d} 
  \hat{\phi}(\ell) 
  \frac{d^k}{dt^k}P_{\ell}^{(\lambda_d)}(t),
$
is absolutely convergent, and, hence, equals $\frac{d^k}{dt^k}\phi (t)$.
\end{proof}
When $1\le p< \infty$, 
the smoothness spaces we consider are the Sobolev (for integer smoothness) and Besov classes which we denote by  $W_p^k(\sphere)$ and $B_{p,\infty}^s(\sphere)$, respectively. 
For $p=\infty$, we consider $C^k(\sphere)$ and the Besov classes $B_{\infty,\infty}^{s}(\sphere)$.  
These can be defined on $\sphere$ in several, equivalent, customary ways. 
The simplest way to define Sobolev spaces is to use a partition of unity and local changes of variables to import the definition from $\reals^d$ as in \cite[Sect. 3]{LNWW}. See the reference \cite{Tri} for this and other definitions. 
Of principal importance to us is the fact that
 $L_p(\sphere) = W_p^0(\sphere)$ and
 $\Delta$ boundedly maps $W_p^s(\sphere)$ to $W_p^{s-2}(\sphere)$ (for $s\ge 2$).
 We postpone the discussion of Besov spaces until Section 6.

\section{Polyharmonic Kernels and Surface Splines}
The kernels we introduce in this section, the polyharmonic kernels, are fundamental solutions for certain elementary partial differential operators. 
In Section 3.1, we begin by defining the kernels in terms of the operators they invert.
This indirect approach is taken because it is key to understanding the approximation scheme discussed in subsequent sections. 
A more direct expression in terms of Gegenbauer polynomials, (\ref{green_expansion}), is also given.

Lemma \ref{GreensFns} provides an asymptotic expansion 
$G_m\sim \sum_{j=0}^{\infty}\gamma_j \phi_{s+j}$  
of polyharmonic kernels in terms of simpler kernels, called surface splines. 
This is developed in Section 3.3. 
In the course of demonstrating the asymptotic expansion, we make the complementary observation, Lemma \ref{ss_is_Green}, that the surface splines are polyharmonic kernels. 
This is the focus of Section 3.2. 

\subsection{Polyharmonic Kernels}
\begin{definition}[Polyharmonic Kernels] \label{ker_def}
Let $m>d/2$ be an integer. For 
 $r_1,\dots,r_m\in \comps$, 
the polyharmonic kernel $G_m=G(\, \cdot\, ;r_1,\dots,r_m)$, 
defined on $[-1,1)$, is the fundamental solution for the  product of perturbed Laplace -- Beltrami operators  
$(\Delta-r_1)\dots (\Delta-r_m)$.
\end{definition}

Our interest in polyharmonic kernels stems from certain integral identities they  satisfy. Such identities may hold for a general kernel $\k_{M}$ (not necessarily polyharmonic, or even zonal),
\begin{equation}\label{rep}
f(x) = 
\int_{\mathbb{S}^d} 
  \L_{M} (f-p_f)(\alpha) \k_{M}(x\cdot\alpha)\, 
\dif \alpha +p_f(x).
\end{equation}
where $\L_{M}$ is a differential operator of order $M$ 
whose nullspace is contained in the finite dimensional space 
$\Pi_{\mathcal{J}} = \sum_{j\in \mathcal{J}}\mathcal{H}_{j}$ 
of spherical harmonics of prescribed degrees $j \in \mathcal{J},$ 
and where 
$
p_f 
= 
\sum_{j\in \mathcal{J}} 
\sum_{m=1}^{N(d,j)} 
  \langle f, Y_{j,m}\rangle Y_{j,m}
$ 
is the ($L_2$) orthogonal projection onto this space. 
Because $\Pi_{\mathcal{J}}$ is finite dimensional, 
$\|p_f\|_{X} \le \const(X,p,\mathcal{J}) \|f\|_p$
for any norm $\|\cdot\|_X.$ 
Hence, the identity (\ref{rep}) extends, by continuity, to every space $W_p^{M}(\sphere)$, with $1\le p < \infty$.
%%%%%%%%%%%%%%%%%%%%%%%%%%%%%%%%
\begin{definition}
If (\ref{rep}) holds for all $f\in C^{M}(\sphere)$,  then $\k_{M}$ is said to satisfy an integral identity of order $M$.
\end{definition}
%%%%%%%%%%%%%%%%%%%%%%%%%%%%%%%%
When $\k_{2m} = G(\cdot ; r_1,\dots,r_m)$, the operator is  
$\L_{2m} = (\Delta -r_1)\dots(\Delta-r_m)$, 
and $\mathcal{J}$ must at least capture the indices corresponding to the eigenvalues used to construct $\L_{2m}$. 
That is, $\mathcal{J}$ contains each index $j\in \nat$ for which  there is $r_{\ell}$ of the form $r_{\ell} = j(j+d-1)$ (there will be at most $m$ such indices, although the set $\mathcal{J}$ is free to  contain more). 
Thus every kernel $G(\cdot ; r_1,\dots,r_m)$ satisfies an integral identity of order $2m$.

We now show that the polyharmonic kernels can be decomposed as 
linear combinations of  {\em surface splines} (perhaps more accurately called ``restricted surface splines''), which are zonal functions
\[\phi_s(t) := \begin{cases}(1- t)^s\; \log(1-t) & \text{for}\  s\in \nat;\\
                (1- t)^s & s\in \nat-\frac{1}{2} 
              \end{cases}\]
Roughly, these are restrictions to the sphere of a well known family of RBFs: the surface splines,  $|\cdot|^{\beta}$ and $|\cdot|^{\beta} \log |\cdot|$, produce the fundamental solution of the $(\beta +d)/2$-fold Laplacian in $\reals^d$. The zonal kernels considered here are restrictions of such to the sphere, by way of the identity $\frac12 (x-\alpha)^2 = 1 - x\cdot \alpha$.   
Providing this decomposition is important to determining error estimates, because there are precise bounds for the surface splines and their derivatives, especially near the singularity $x=\alpha$.  
For $t\ge0$ it is not difficult to see that there exist constants $\beta_{s,j}$ so that
\begin{equation}\label{phi_diff}
 \left|\phi_s^{(j)} (t)\right|=  
\beta_{s,j}(1-t)^{s - j} \quad\text{for $j>s$}.
\end{equation}
 %secondary goal of this section is to show 

Our investigation of polyharmonic kernels 
begins with observing their expansions in Gegenbauer polynomials. 
The series expansion for $G_m$ follows by Fourier inversion; 
its Fourier coefficients are obtained by reciprocating the 
symbol of the differential operator that $G_m$ inverts:
$$ G_m(x\cdot\alpha) = \sum_{\ell=1}^{\infty} \frac{\ell+\lambda_d}{\omega_d \lambda_d} \prod_{j=1}^m  
[\ell(\ell+d-1)-r_j]^{-1}P_{\ell}^{(\lambda_d)}(x\cdot \alpha).$$  
It is often useful to adopt the notation $\l  := \l(\ell,d) := \ell+\lambda_d$, in which case the Gegenbauer expansion becomes 
\begin{equation}\label{green_expansion}\
 G_m(x\cdot\alpha) = \sum_{\ell=1}^{\infty} \; \frac{\ell+\lambda_d}{\omega_d \lambda_d} \;
\left[\prod_{j =1}^m \frac{1}{[\l^2 - \lambda_d^2]-r_{j}}\right]
P_{\ell}^{(\lambda_d)}(x\cdot \alpha).
\end{equation}

\subsection{Surface Splines}
The series expansion for surface splines is more difficult. It has been studied recently in \cite{BaHu} and \cite{OdLe}. % \cite{NSW}. 
These results allow a precise expansion of the kernel in Gegenbauer polynomials.
\begin{lemma}\label{ss_coeff}
For $s\in \nat/2$ satisfying $m:=s+d/2\in \nat$, and $\l=\ell+\lambda_d$, there is a nonzero constant $C_s$ ( depending on $s$ and $d$) such that the Fourier coefficient is
$$
  \widehat{\phi_s}(\ell) 
  = 
  C_{s} 
  \prod_{\nu=1}^m [\l^2 - (\nu- \frac12)^2]^{-1}
$$
for $\ell>s$ when $d$ is even, and for all $\ell$ when $d$ is odd.
\end{lemma}
\begin{proof}
The formula 
$\phi_s(x\cdot\alpha) 
= 
\sum_{\ell=0}^{\infty} 
  \; a_{\ell} \; P_{\ell}^{(\lambda_d)}(x\cdot \alpha)
$ 
holds with 
$$a_{\ell} 
= 
C_{s}\; 
\frac{\ell+\lambda_d}{\omega_d \lambda_d} \; 
\frac{\Gamma(\ell-s)}{\Gamma(s+\ell+d)}
$$  
for $\ell>s$  when $s\in \nat$ by \cite{BaHu}[(2.20)] and for all $\ell$ when $s\in \nat-1/2$ by \cite{BaHu}[(2.12)].
Utilizing the notation 
$\l= \ell+\lambda_d$ and $\s = s+ \lambda_d$ 
(and noting that $\s$ is in $\nat-1/2$, 
since we assume that $s+d/2$ is an integer), 
the factor $\Gamma(\ell-s)/\Gamma(s+\ell+d)$ 
simplifies to 
$\bigl[(\l^2 - (\frac{1}{2})^2)\cdots  (\l^2 - \s^2)\bigr]^{-1}$, 
and the lemma follows.
\end{proof}

Thus, for any positive half-integer $s$, we have the expansion for surface splines: %the %Legendre series
 $\phi_s(x\cdot\alpha) 
= 
p(x\cdot \alpha)
+
C_{s}
\sum_{\ell=0}^{\infty}  
 \frac{\ell+\lambda_d}{\omega_d \lambda_d} \; 
\prod_{\nu=1}^m [\l^2 - (\nu- \frac12)^2]^{-1}
 % \hat{\phi_s}(\ell) \;
  P_{\ell}^{(\lambda_d)}(x\cdot \alpha)
$, 
%with 
%$\hat{\phi_s}(\ell) = C_{s,d}\prod_{\nu=1}^m [\l^2 - (\nu- \frac12)^2]^{-1}$
although the extra polynomial term $p\in \Pi_{s}[-1,1]$ is only needed when $d$ is even.
%we are left with the Legendre series
%\begin{equation} \label{ss_series}
%\phi_s(x,\alpha) = P(x\cdot \xi)+ 
%\sum_{\ell=0}^{\infty} \; (-1)^{s+1/2} \; C_{s,d}\; \frac{\ell+\lambda_d}{\omega_d \lambda_d} \; 
% [(J^2 - (\frac{1}{2})^2)\cdots  (J^2 - S^2)]^{-1} \; P_{\ell}^{(\lambda_d)}(x\cdot \xi)
%\end{equation}
%GreensFns
%
%
\begin{lemma} \label{ss_is_Green} Let $m =  s+d/2$. The kernel 
$(x,\alpha)\mapsto \phi_s(x\cdot \alpha)$ satisfies an integral identity of order $2m$ with operator
%\[\L_{2m} = \prod_{\nu+d/2=1}^m (\Delta - \nu(\nu +d -1)),\]
\[\L_{2m} = \prod_{j=1}^m \bigl[\Delta - (j-d/2)(j +d/2 -1)\bigr] ,\]
 and $p_f = \sum_{\ell \le s} \sum _{m=1}^{N(d,\ell)} \langle f,Y_{\ell,m}\rangle Y_{\ell,m}$, the projection onto $\Pi_{2m-d}$.
\end{lemma}
\begin{proof}
Since the symbol of the Laplacian is $\l^2 -\lambda_d^2$, 
%we observe that we can write 
the $\nu^{\mathrm{th}}$ factor in the denominator of the Gegenbauer coefficient of $\phi_s$ %as
is
$$
\l^2 - (\nu- \frac12)^2 
= 
\l^2 -\lambda_d^2 - 
\bigl[(\nu- \frac12)^2-\lambda_d^2\bigr]
= 
\l^2 -\lambda_d^2 - (\nu- \frac{d}{2})(\nu + \frac{d}{2}-1).
$$ 
Thus, when $d/2$ is fractional, $\phi_s$ is the
fundamental solution for the invertible differential operator whose symbol is $\prod_{j=1}^m \left[(\l^2-\lambda_d^2)- (j- \frac{d}{2})(j+\frac{d}{2}-1)\right]$, since the eigenvalues of $\Delta$ are integers (the integers $k(k+d-1)$). 
When $d/2$ is integral, the differential operator is invertible on the complement of the space of spherical harmonics of degree less than or equal to $s$.
\end{proof}
\subsection{Surface Spline Expansion of Polyharmonic Kernels}
%
%:Lemma:GreensFns
%
\begin{lemma}\label{GreensFns}
For positive integers $m$ and $d$, let $s=m-d/2$. 
The polyharmonic kernel 
$(x,\alpha)\mapsto G_m(x\cdot\alpha;r_1,\dots,r_m)$ %corresponding to the $m$-fold Laplacian on $\reals^d$ 
can be written as
\[G_m(x\cdot\alpha) =  \sum_{j=0}^{J-1} \gamma_j  \, \phi_{s+j}(x\cdot\alpha) + R_{J}(x\cdot\alpha) \]
with $R_{J}\in C^{(J+s-\epsilon)}\bigl([-1,1]\bigr)$.
\end{lemma}
\begin{proof}
We begin by expanding each of the Fourier coefficients of $G_m$. From (\ref{green_expansion}) we observe that 
$\widehat{G_m}(\ell) = \prod_{j=1}^m (\l^2 - \lambda_d^2 - r_j)^{-1}$. Factoring
$\l^{-2m}$, we have, for  $\ell>\max\bigl(\sqrt{|r_1|},\dots,\sqrt{|r_m|}\bigr)$, that
\[ \widehat{G_m}(\ell) %= \prod_{j=1}^m (\l^2 - \lambda_d^2 - r_j)^{-1}  = 
= \l^{-2m} \prod_{j=1}^m \left(1 - \frac{\lambda_d^2 + r_j}{\l^2}\right)^{-1}
 =  \l^{-2m} \left(1 + \sum_{n=1}^{\infty}A_n \l^{-2n}\right).
\]
The second equality follow by writing each factor in the product as a Neumann series
(i.e., a series of the form $(1-a)^{-1} = \sum_{j=0}^{\infty} a^j$), and then by multiplying the $m$ series.
We do likewise for the coefficients of $\phi_{s+j}$ (determined in Lemma \ref{ss_coeff}) when $\ell> s+J$: 
\begin{eqnarray*} 
\widehat{\phi_{s+j}}(\ell) &=& 
%{C_{s+j}}\l^{-2(m+j)}\prod_{\nu=1}^{m+j} \left(1 - \frac{(\nu- \frac12)^2}{\l^2}\right)^{-1}
%\\
%&=&
{C_{s+j}}\l^{-2(m+j)}\left( 1+\sum_{n=1}^{\infty}B_{j+n,j} \l^{-2n}\right).
\end{eqnarray*}
This allows us to choose the coefficients $\gamma_0,\gamma_1,\dots$ in succession, via
 $\gamma_j := C_{s+j}^{-1}(A_j - \sum_{k=0}^{j-1}\gamma_k B_{j,k})$.
 With this choice, the first $J$ terms in the asymptotic expansion of $\widehat{R_J}(\ell)$ are forced to vanish. The fact that each $\gamma_j$ depends only on the previous coefficients $\gamma_k$, $k<j$, is evident from Table \ref{Neumann}.
\begin{table}[ht]
\[  \begin{array}{llrrll}
%
%
%First
\widehat{G_m}(\ell) 
%&= \displaystyle\prod_{j=1}^m (\l^2 - \lambda_d^2 - r_j)^{-1}
& = 
&\l^{\mbox{}-2m}\bigl(
&1 \, + \, A_{1\phantom{1,}}\l^{\mbox{}-2} 
+ \, A_{2\phantom{1,}}\l^{\mbox{}-4}
+ \, A_{3\phantom{1,}}\l^{\mbox{}-6}
&\!\!\! 
+\,\dots
\bigr)\\
%
%
%
%Second
\widehat{\phi_{s}}(\ell)
%&=
%\displaystyle\prod_{\nu=1}^{m} \left(1 - \frac{(\nu- \frac12)^2}{\l^2}\right)^{-1}
& = 
&{C_{s}}\l^{\mbox{}-2m}\bigl(
&  
1 \,+\, B_{1,0}\l^{\mbox{}-2}  
+\, B_{2,0}\l^{\mbox{}-4} 
+\, B_{3,0}\l^{\mbox{}-6}
&\!\!\! 
+\,\dots
\bigr)\\
%
%
%
%Third
\widehat{\phi_{s+1}}(\ell)
%&=
%\displaystyle\prod_{\nu=1}^{m+1} \left(1 - \frac{(\nu- \frac12)^2}{\l^2}\right)^{-1}
&= 
&C_{s+1}\l^{\mbox{}-2m} \bigl(%B_{1,1}J^{\mbox{}-2m-4} 
&\l^{\mbox{}-2} +\, B_{2,1}\l^{\mbox{}-4} 
+\, B_{3,1}\l^{\mbox{}-6} 
&\!\!\! 
+\,\dots
\bigr)
\\
%
%
%
%Fourth
\widehat{\phi_{s+2}}(\ell)
%&=
%\displaystyle\prod_{\nu=1}^{m+2} \left(1 - \frac{(\nu- \frac12)^2}{\l^2}\right)^{-1}
&= 
&C_{s+2}\l^{\mbox{}-2m}\bigl(
&\l^{\mbox{}-4} 
+\, B_{3,2}\l^{\mbox{}-6} 
&\!\!\!  
+\,\dots
\bigr)
\\
&&\dots&&\\
\end{array}\]
\caption{Expansion of Gegenbauer coefficients}\label{Neumann}
\end{table}
The coefficients of the remainder term are determined to be 
$\widehat{R_J}(\ell) =\widehat{G_m}(\ell) - \sum_{j=0}^{J-1} \gamma_j.$ %\widehat{\phi_{s+j}}(\ell) \le \mathrm{const}(J) \l^{-2(J+m)}\]
and $$|\widehat{R_J}(\ell)| \le \sum_{k=0}^{\infty} \left|A_{J+k} - \sum_{j=0}^{J-1} \gamma_j B_{J+k-j-1,j}\right|\l^{\mbox{}-2(m+J+k)}\le\const(J)< \infty$$ for sufficiently large $\l$. 
%This follows from the fact that the power series (in $\l$) $\sum_{k=0}^{\infty} A_k \l^{\mbox{} - 2k} = [1- (\lambda_d/\l)^2]^{-m}$ can be written as an $m$-fold product of Neumann series with interval of convergence $(-\lambda_d, \lambda_d)$, and, similarly, each power series  $\sum_{k=0}^{\infty} B_{k,j} \l^{\mbox{} - 2k}$ can also be written as a product of Neumann series. 
By Proposition \ref{smoothness}, the lemma follows.
\end{proof}
%
%
%
% \begin{remark} We note that the preceding lemma holds for Green's functions for operators of the form $(\Delta - r_1)\cdots (\Delta - r_m)$ with minor modifications (in fact, the Neumann series for the Fourier coefficient of such a Green's function is nearly identical to that of $\hat{G_m}(\ell)$.)
% \end{remark}

\section{Replacing the Kernels I: Finding the Coefficients}
We now wish to investigate a `coefficient kernel' $\a:\Xi\times \sphere\to \reals$
that will allow us to effectively replace a $\kernel(x\cdot\alpha)$ with  
$\sum_{\xi \in \Xi} \a(\xi,\alpha) \kernel(x\cdot\xi)$
in the representation (\ref{rep}). To
do so, the {\it exchange} $\e_{\kernel}$ given by:
\[\e_{\kernel}(x, \alpha):=|\kernel(x\cdot\alpha) - \sum_{\xi \in \Xi} \a(\xi,\alpha) \kernel(x\cdot\xi)|\] 
must be appropriately small in $L_{\infty}$, and it must decay away from $\alpha = x$. %This is an approximation problem in its own right; because the kernels considered in this paper are in $\mathrm{Lip}(2m-d)$ (in the sense that $x \mapsto \k_m(x\cdot\alpha)$ is in $\mathrm{Lip}(2m-d)$ for every $\alpha$, although they are infinitely smooth away from $x=\alpha$), the {exchange} cannot be expected to have smaller $L_{\infty}$ error than $\rho^{2m-d}$ (where $\rho$ measures the density of the centers), at least in a neighborhood of $\alpha = x$. 
The remarkable thing is that this can %actually 
be achieved using only a fixed number of centers near to the singularity. In this section, we develop a technique for choosing coefficients $\a(\xi,\alpha)$ that -- in the following section -- is shown to provide an appropriately small {exchange}.

The two key quantities we need to resolve are the spherical harmonic {\em precision} 
(the degree of 
spherical harmonics reproduced by the coefficient kernel) and the
rate of decay of the error as $\dist(x,\alpha)$ increases. % reflected in the exponent $\nu$ 
%(appearing in the estimate (\ref{est})). 
As in the Euclidean setting, these are related: the higher the degree of spherical harmonic precision, the more rapidly the {exchange} decays away from the singularity. %: ultimately we need to have $\nu=d+1$,
%because the integral of (\ref{est}) is roughly
%$\int\rho^{2m-d}(1+r/\rho)^{-nu}r^{d-1}\dif r \le \rho^{2m} \int_1^{\infty} r^{d-\nu -1} \dif r$, which is finite when $\nu>d$.  It is shown,

% In order to show (\ref{est}), we begin by making
% some explicit conditions on the coefficient kernel.

\begin{definition}[CKC] For a set of centers $\Xi \subset \sphere$ the kernel $\a:\Xi\times \sphere \to \reals$  satisfies the Coefficient Kernel Conditions (or CKC) with precision $L$, radius $\rho$ and stability $K$ if 
it is measurable and the following three conditions hold:\\[.5ex]
{\setlength{\parskip}{0.5ex}\vspace{-2.75ex}\begin{description}
\item[\bf CKC 1 (Support)] $\a(\xi,\alpha) = 0$ when $\dist (\xi,\alpha)>\rho$.
\item[\bf CKC 2 (Precision)]
For $S \in \Pi_L$, $\sum_{\xi\in \Xi} \a(\xi,\alpha) S(\xi) = S(\alpha)$
\item[\bf CKC 3 (Stability)] %he kernel is uniformly in $\ell_1$: t
 $\max_{\alpha\in\sphere}\sum_{\xi \in \Xi} |\a(\xi,\alpha)|\le K$.
\end{description}}
\end{definition}
Such a local reproduction property always holds for sufficiently dense centers. This is demonstrated in the following lemma.
\begin{lemma}\label{dense}
Given a precision $L$ and centers $\Xi$ having density $h<h_0$ (with $h_0$ determined by
$L$), there exists a coefficient kernel $\a:\Xi\times\sphere\to \reals$ satisfying the CKC with radius $\rho = 48L^2 h$ and stability $K=2$.
\end{lemma}
%%%
%
%%%
\begin{proof}
Let $\Xi_{\alpha} := \Xi\cap C(\alpha,\rho)$, the set of centers a distance $\rho$ from $\alpha$.
Following what is, by now, a fairly standard technique in scattered data approximation (originally developed for the sphere in \cite{JSW}, and deftly exposited in \cite[Ch. 3]{Wend}), a coefficient kernel is shown to exist if the sampling operator
$$R_{\alpha} :  
\Pi_{L} \to (\Pi_{L})_{|_{\Xi_{\alpha}}}:  
 p \mapsto  p_{|_{\Xi_{\alpha}}}
$$
is boundedly invertible when the domain and range are endowed with the 
$L_{\infty}$ and $\ell_{\infty}$ topologies, respectively. To be precise, we 
must show that the norm of the inverse of the sampling operator is bounded by 2: 
$\|R_{\alpha}^{-1}\| \le 2$, which is accomplished in Lemma \ref{sampling}, below.
Bounded invertibility of $R_{\alpha}$ implies that the norm of the adjoint 
$$
\left(R_{\alpha}^{-1} \right)': 
\Pi' \to   
\left((\Pi_{L})_{|_{\Xi_{\alpha}}}\right)'
$$ 
is similarly bounded. By the Hahn-Banach theorem, there is a norm-bounded extension
of the functional 
$\left(R_{\alpha}^{-1} \right)' \delta_{\alpha}\in \bigl((\Pi_{L})_{|_{\Xi_{\alpha}}}\bigr)'$
 in the space, $\bigl(\ell_{\infty}(\Xi_{\alpha})\bigr)'$.
This can be viewed as an element of $\ell_1\bigl(\Xi_{\alpha}\bigr),$ and, by zero extension,
it is in $\ell_1\bigl(\Xi \bigr).$
We call this sequence $\a(\cdot,\alpha)$ and note that its $\ell_1$ 
norm is bounded by 
$\bigl\| \left(R_{\alpha}^{-1} \right)'\bigr\| \|\delta_{\alpha} \| \le 2$. 

The measurability of the kernel $\a$ is a consequence of its piecewise continuity, which
we now demonstrate.
% on a finite collection of simple sets. 
For each $\upsilon\subset \Xi$, we define the open set
$\Omega_{\upsilon}:=\cap_{\xi\in \upsilon} C(\xi,\rho).$ These can be refined to a (finite) collection of sets
$$\widetilde{\Omega}_{\upsilon}
:=
\Omega_{\upsilon}\setminus \bigcup_{\upsilon\subsetneq \varsigma} \Omega_{\varsigma}$$
that partitions $\sphere$. 
For each $\alpha$ in $\widetilde{\Omega}_{\upsilon}$, 
the sampling operators $R_{\alpha}$ share a common target 
$\Pi_{|_{\upsilon}}$, and the operator valued map
$\alpha \mapsto R_{\alpha}$ is well defined and  Lipschitz. Indeed, 
$\|R_{\alpha}p - R_{\alpha'} p\| \le C |\alpha - \alpha'| \|\nabla p\|_{\infty} \le C_L |\alpha - \alpha'| \|p\|_{L_{\infty}}$
implies that $\|R_{\alpha} - R_{\alpha'}\| \le C_L |\alpha - \alpha'|.$
The inverse is similarly Lipschitz, because
$R_{\alpha}^{-1} - R_{\alpha'}^{-1} = 
R_{\alpha}^{-1}\bigl[ R_{\alpha'} -   R_{\alpha}\bigr] R_{\alpha'}^{-1}.$
For $\alpha \in \widetilde{\Omega}_{\upsilon}$ 
the family of sequences $\a(\cdot,\alpha)$ have their support in $\upsilon$. 
To show that $\alpha \mapsto \a(\cdot,\alpha)$ is continuous we simply observe that
\begin{eqnarray*}
 \|\a(\cdot,\alpha) - \a(\cdot,\alpha')\|_{\ell_1} 
 &\le&
  %  \left\| \bigl(R_{\alpha}^{-1} - R_{\alpha'}^{-1}\bigr)' \right\| 
   % \|\delta_{\alpha}\| 
    \left\| R_{\alpha}^{-1}%\bigl[ R_{\alpha'} -   R_{\alpha}\bigr]
- R_{\alpha'}^{-1}\right\|
    \|\delta_{\alpha}\| 
    +
    \left\|\bigl(R_{\alpha'}^{-1} \bigr)'\right\| 
    \|\delta_{\alpha'} - \delta_{\alpha}\| \\
 &=&
%C_{L} |\alpha - \alpha'| = 
O(|\alpha - \alpha'|).
 \end{eqnarray*}
\end{proof}
%The operator valued
%map $\alpha \mapsto \left(R_{\alpha}^{-1} \right)' $ is differentiable, and $\alpha \mapsto \delta_\alpha$
%is Lipschitz.

\begin{lemma}\label{sampling} Given a precision $L$ and centers $\Xi$ with density $h<h_0$,  let $\rho =48 L^2$ and  $\Xi_{\alpha} := \Xi\cap C(\alpha,\rho)$ for each $\alpha\in\sphere$.
The sampling operator $R_{\alpha}$ is boundedly invertible on the space of spherical harmonics of degree $L$ or less, and
$$\|p\|_{L_{\infty}(C(\alpha,\rho))} 
\le 
2 \|R_{\alpha} p\|_{\ell_{\infty}(\Xi_{\alpha})}.$$
\end{lemma}
\begin{proof}
This is accomplished by noting that spherical harmonics, when restricted to great circles, are trigonometric polynomials. From this we can apply the Markov inequality of Videnski\u\i \cite{Vide}, which states that for a trigonometric polynomial, $\tau$ of degree $n$
\begin{equation*}
|\tau'(\theta)| \le 2 n^2 \cot (\omega/2) \|\tau\|_{L_{\infty}(-\omega,\omega)}\quad \text{for $\omega<\pi$, $|\theta|\le \omega$}
\end{equation*}
to control the size of a spherical harmonics having many zeros in a spherical cap.

Select $p\in \Pi_{L}$ and find $x_0$ such that $|p(x_0)| = \|p\|_{L_{\infty}(C(\alpha, \rho))}.$
Following Wendland \cite[p.30]{Wend}, we take $\xi \in \Xi\cap  C(\alpha, \rho)$ so that $\xi$ is in a cone with vertex $x_0$ and distance from $x_0$ less than $h + \frac{h}{\sin\theta}\le 3h$ (this is possible because a cap of radius $h$ with center located at a distance $\frac{h}{\sin \theta}$ from $x_0$ is contained in the cone of aperture $\theta$). Let $\hat{x}_0$ be the terminal point of the geodesic segment starting at $x_0$, passing through $\xi$ and having length $\rho$. Restricting $p$ to this geodesic gives a trigonometric polynomial of degree $L$. Vis., there is $q:[-\pi,\pi] \to \comps$,  $q(\theta) = \sum_{|j|\le \order} a_j e^{ij\theta}$, such that $q(-\rho/2) = p(x_0)$ and $q(\rho/2) = p(\hat{x}_0)$ and
\[|p(x_0) - p(\xi)| \le  \int_{-\rho/2}^{-\rho/2 +|x_0 - \xi|} |q'(t)| \, \dif t.\]
By Videnski\u\i's Markov inequality, $|q'(t)|\le 2 L^2 \cot(\rho/4)\|q\|_{L_{\infty}(-\rho/2,\rho/2)}$,
and, consequently, we have that
\[|p(x_0) - p(\xi)| \le 3h\,\, 2 (L)^2\,\, \frac{4}{\rho} \|p\|_{L_{\infty}(C(\alpha,\rho))}\le \frac{1}{2}\|p\|_{L_{\infty}(C(\alpha,\rho))}.\]
Thus $\|p\|_{L_{\infty}\bigl(C(\alpha,\rho)\bigr)}\le 2 \|p_{|_{\Xi}}\|_{\ell_{\infty}}$ 
and the lemma is proved.
\end{proof}

A consequence of the CKC is that for any $x$ 
and any zonal function
$\k$ that is smooth on the interval $\mathcal{I}_x:=[\min Q_x,\max Q_x]$, where 
$Q_x = \{x\cdot \alpha\}\cup \{x\cdot\xi:\xi\in (\Xi\cap C(\alpha,\rho)\}$,
the exchange can be estimated in terms of the length of the interval
$\mathcal{I}_x$ and the size of derivatives of $\k$ {\em purely on $\mathcal{I}_x$}.  This is the point of the following lemma:
%
%:Lemma Joe_Est
%
\begin{lemma}\label{Joe_est_lemma}
Given a coefficient kernel satisfying the CKC with precision $L$, if $\k \in C^{(L+1)}(\mathcal{I}_x),$  then the {exchange} satisfies
\begin{equation}\label{Joe_est}
\e_{\k}(x,\alpha)\le \frac{\|\a(\cdot,\alpha)\|_{\ell_1}}{L\;!} \max_{\xi \in \Xi\cap C(\alpha,\rho)}\, |x \cdot (\alpha - \xi)|^{L+1} \, \max_{t\in \mathcal{I}_x} |{\k}^{(L+1)}(t)| 
\end{equation}
\end{lemma}
\begin{proof}
Let both $x$ and $\alpha$ be fixed, set $x \cdot \alpha = t_{\alpha} \in [-1, 1]$ and choose the Taylor polynomial of degree $L$, $q_{L,t_{\alpha}}$, of $\k$ expanded about $t_{\alpha}$. Now $q_{L,t_{\alpha}}$ may be rewritten as a linear combination of Gegenbauer polynomials,
\[
q_{L,t_{\alpha}}(t) =  \sum_{\ell=0}^L \widehat{q_{L,t_{\alpha}}}(\ell) \frac{\ell+\lambda_d}{\omega_d \lambda_d}P_{\ell}^{(\lambda_d)}(t).
\]
Note, furthermore, that 
$q_{L,t_{\alpha}}(x\cdot \alpha)- \sum \a(\xi,\alpha)q_{L,t_{\alpha}}(x\cdot\xi) = 0$ 
by the addition theorem, since $q_{L,t}(x\cdot \zeta) 
= 
\sum_{\ell}\widehat{q_{L,t}}(\ell) 
\sum_m Y_{\ell,m}(x)Y_{\ell,m}(\zeta)$ 
and each 
$Y_{\ell,m}(\zeta)$ 
is annihilated by 
$\mu = \delta_{\alpha}-\sum \a(\xi,\alpha) \delta_{\xi}$.  
Consequently, 
$\e_{\k}(x,\alpha) 
\le 
\sum_{\xi} |\a(\xi,\alpha)| \:|\k(x\cdot \xi) - q_{L,t_{\alpha}}(x\cdot \xi)|$, 
and the 
Taylor's theorem gives: 
\[|\k(t_{\xi}) - q_{L,t_{\alpha}}(t_{\xi})| 
\le \frac{1}{L!} |t_{\xi}- t_{\alpha}|^{L+1} \max_{u\in \mathrm{co}(t_{\xi},t_{\alpha})}|\k^{(L+1)}(u)|.\]
\end{proof}

\section{Replacing the Kernels II: Estimates}
Having found coefficients suitable for replacing the kernel in a representation (\ref{rep}), we now obtain estimates on the {exchange} in an effort to estimate the norm of the operator $E_{\k}:L_p(\sphere)\to L_p(\sphere)$, 
defined by $E_{\k} g(x) = \int \e_{\k}(x,\alpha) g(\alpha)\dif \alpha$. 
The bound, $\|E_{\k}\| := \sup_{0\ne g\in L_p}\|E_{\k}g\|_p/\|g\|_p$  %on such an operator 
gives us essentially 
the error 
estimates we desire, since the pointwise error, 
$\mathcal{E}(x) := |f(x) - \int_{\sphere}\sum \a(\xi,\alpha) \k(x\cdot\xi)\L_{2m}f(\alpha) \dif \alpha|$, 
satisfies, by (\ref{rep}), 
\begin{eqnarray*}
\mathcal{E}(x) &=&  
\left|
  \int_{\sphere} 
    \left(\k(x\cdot \alpha) - \sum \a(\xi,\alpha) \k(x\cdot\xi) \right)\L_{2m}f(\alpha) 
  \dif \alpha
\right|\\
&\le& \int_{\sphere} \e_{\k}(x,\alpha) |\L_{2m}f(\alpha)| \dif \alpha.
\end{eqnarray*}
In other words,
$
%\left(
%  \int_{\sphere}
%    \bigl|
%      \int_{\sphere}  \e_{\k_m}(x,\alpha) |\L_{2m}f(\alpha)| \dif \alpha
%    \bigr|^p
%  \dif x 
%\right)^{1/p}
\|\mathcal{E}\|_p 
\le 
\bigl\|E_{\k} |\L_{2m}f |\bigr\|_{p}
 \le \|E_{\k}\| \|\L_{2m}f\|_p.$
Because of the expansion from  Lemma \ref{GreensFns}, we focus on obtaining the estimates for surface splines first, before moving to polyharmonic functions in general. %11/24 Changed from: before moving the to Green's functions
%
%:Lemma: ss_loc
%
\begin{lemma}\label{ss_loc}
Let $m=s+d/2$. Assume $\a$ is a coefficient kernel satisfying the CKC with radius $\rho$,  precision $2m$ and stability $K$. Then for 
%$j= 0,\dots, %2m-s-1
%(d-1)/2$, 
$j\in \mathbb{N}$
the {exchange} of the kernel $\phi_{s+j}$ satisfies 
\begin{equation*}
\e_{\phi_{s+j}}(x,\alpha) \le \const(K,m,j,d)\,
%%%%%%%%%%%%%%%%%%%%%%%%%%%%%%
% \begin{cases}
% \rho^{2(m+j)-d}\left(1+ \frac{\dist (x,\alpha)}{\rho}\right)^{2j-d-1},& \quad 
% \dist (x,\alpha)\le \pi/2; \\
% \rho^{2m+1},& \quad  \dist (x,\alpha)> \pi/2.
% \end{cases}
%%%%%%%%%%%%%%%%%%%%%%%%%%%%%%
\rho^{2(m+j)-d}\left(1+ \frac{\dist (x,\alpha)}{\rho}\right)^{2j-d-1}
%%%%%%%%%%%%%
%\rho^{2(m+j)-d}
%\max\bigl( \left(1+ \frac{\dist (x,\alpha)}{\rho}\right)^{2j-d-1}, \rho^{d+1-2j}\bigr)
%%%%%%%%%%%%%
\end{equation*}
\end{lemma}
\begin{proof}
 We consider three regions, for a  fixed `north pole' $\alpha\in\sphere$:\\[.5ex]
 {\setlength{\parskip}{0.5ex}\vspace{-2.75ex}
\begin{itemize} 
\item[] $\Omega_1 :=\{x\mid\pi/2<\dist(x,\alpha)\le \pi\}$, where the surface spline is smooth;
\item[] $\Omega_2 := \{x\mid0<\dist(x,\alpha)\le \rho\}$, a cap of radius $\rho$ near the north pole;
\item[] $\Omega_3 := \{x\mid\rho<\dist(x,\alpha)\le\pi/2\}$: a band % between $\Omega_1$ and $\Omega_2$ 
where high order derivatives decay.
\end{itemize}}
%\begin{proof}We consider three cases (one for each $\Omega_j$). 

{\bf  $\boldsymbol{\Omega_1}$ :} 
We note that outside the spherical cap 
$C(\alpha,\pi/2)$ the $(2m+1)^{\mathrm{st}}$ derivatives of $\phi_{s+j}$ are
bounded by $\beta_{s+j,2m+1}$, and we can use (\ref{Joe_est}) to obtain 
$\e_{\phi_{s+j}}(x, \alpha) \le \rho^{\val+1}  \frac{K}{2m!}\beta_{s+j,2m+1}$.

{\bf  $\boldsymbol{\Omega_2}$ :} 
In the cap nearest to $\alpha$,
we use the fact that $\phi_{s+j}$ has a high order zero. 
Here we need a relationship (used later, as well) between the geodesic distance of two points and their inner product
\begin{equation}
1-\frac{1}{2}\left(\dist(x,\zeta)\right)^2\le x\cdot \zeta\le 1-\frac{4}{\pi^2}\left(\dist(x,\zeta)\right)^2\label{dot}\quad\text{for }\mathrm{dist}(x,\zeta)\le \pi/2.
\end{equation}
For even $d$, the proof is complicated by the $\log$ factor, so we consider this case only, as the odd case follows by a similar but much simpler argument. 
We proceed by writing 
\begin{eqnarray*}
&\mbox{}& |1-x\cdot \zeta|^{s+j} \log|1-x\cdot \zeta| \\
&=&\rho^{2(s+j)}\left|\frac{1-x\cdot \zeta}{\rho^2}\right|^{s+j} \log|\rho^2|+ 
\rho^{2(s+j)}\left|\frac{1-x\cdot \zeta}{\rho^2}\right|^{s+j}
\log\left|\frac{1-x\cdot \zeta}{\rho^2}\right|
\end{eqnarray*}
Since $s$ is even when $d$ is even, the first term is simply a spherical harmonic in $\zeta$ (by the addition theorem), of degree $s+j\le2m$, and is therefore annihilated by the functional $\mu = \delta_{\alpha} - \sum_{\xi \in \Xi} \a(\xi,\alpha) \delta_{\xi}.$ Thus, we need only apply $\mu$ to the second term; we obtain (by the left hand side of (\ref{dot})),
\begin{eqnarray*}
\e_{\phi_{s+j}}(x,\alpha)&\le&
 (1+K)\rho^{2(s+j)}\max_{\zeta\in C(\alpha,3\rho)} \left|\frac{1-x\cdot \zeta}{\rho^2}\right|^{s+j}\log\left|\frac{1-x\cdot \zeta}{\rho^2}\right| \\
&\le &
\rho^{2(s+j)}(1+K)
\left[ \max_{0\le t\le 9/2} \left| t\right|^{s+j} \log\left|t\right| \right]
\end{eqnarray*}

{\bf $\boldsymbol{\Omega_3}$ :} 
The estimate (\ref{dot})  bounds the derivatives of $\phi_{s+j}$,  
but in the northern hemisphere, we can achieve better estimates for $|x\cdot (\xi-\alpha)|$, in the sense that this inner product becomes considerably smaller than $\rho$  when $x$ and $\alpha$ are close. 
Decompose
 $C(\alpha,\kappa) \setminus C(\alpha,2\rho)$ {\em en annuli}, and note that
\begin{equation}\label{diff_dot}
\dist (x,\alpha) \le 2^{k} \rho  \quad \Rightarrow \quad |x\cdot(\alpha - \xi)|\le 2^{k+1} \rho^2.
\end{equation}
For $2^{k-1}\rho\le \dist (x,\alpha)\le 2^k\rho$ we
estimate %the right endpoint of 
$\max {\mathcal{I}}_{x}\le 1- \frac{4}{\pi^2}(2^{k-1}\rho)^2$ by (\ref{dot}) to obtain bounds on the $(2m+1)^{\mathrm{th}}$  derivatives of $\phi_{s+j}$ on $\mathcal{I}_x$.  
On the other hand, we can apply (\ref{diff_dot})
to estimate $|x\cdot(\alpha-\xi)|$. Thus,
\begin{eqnarray*}
\e_{\phi_{s+j}}(x,\alpha)
&\le& 
\frac{K}{2m!} \beta_{s+j,2m+1} \, (2^{k+1} \rho^2)^{\miniorder+1} 
\left(\frac{4}{\pi^2}(2^{k-1}\rho)^2\right)^{s+j-\miniorder-1}\\
& =&   
\const(K,m,j,d) \left(2^{k}\right)^{-d-1+2j} \rho^{2s+2j}
\end{eqnarray*}
The first inequality follows from (\ref{Joe_est}) using (\ref{phi_diff}),(\ref{dot}) and (\ref{diff_dot}) while
the second inequality is a consequence of the fact that
$d = 2m -2s$.
\end{proof} 

\section{Main Results}
We are now in a position to prove our main results, that polyharmonic kernels 
and surface splines deliver $L_p$ approximation orders commensurate with the $L_p$ smoothness of the target function, at least up to a (putative) `saturation order': the order of the differential operator that the kernel inverts. We begin by giving `high order' results, for functions of `full' smoothness. Afterwards, we give the lower orders and the corresponding smoothness spaces by means of real interpolation.
\begin{theorem}\label{Main}
Assume the coefficient kernel $\a:\Xi\times \sphere\to \reals$ satisfies CKC with radius $\rho$, precision $2m$ and stability $K$. Assume, further, that the kernel $\k$  provides an integral identity (\ref{rep}) of order $2m$ and can be decomposed as
$\k = \sum_{j = 0}^{2m-s} \gamma_{j} \phi_{s+j} + R$, 
with $s= m - d/2$ and 
remainder $R\in C^{(2m)}[-1,1]$. Then for $f\in W_p^{2m}(\sphere)$, if $1\le p<\infty$, or for $f\in C^{2m}(\sphere)$ when $p=\infty$,
the approximant
\[
T_{\Xi} f(x) 
= 
p_{f}(x) +\sum_{\xi \in \Xi}  A_{\xi} \k(x\cdot\xi),
\]
with coefficients 
$A_{\xi} = 
\int_{\sphere} \L_{2m} (f-p_f)(\alpha)\; \a(\xi,\alpha)  \dif \alpha,
$
%%%%%%%%%%%%%%%%%
$\xi \in \Xi$,
converges to $f$ in $L_p(\sphere)$ with error:
\[\|f - T_{\Xi}f\|_{L_p(\sphere)} \le 
\const(K,\k) \rho^{2m} \|f\|_{W_p^{2m}(\sphere)}\]
 and with coefficients satisfying
$\|A\|_{\ell_1(\Xi)} \le \const(K,\k) \|f\|_{W_p^{2m}(\sphere)}.$
\end{theorem}
The decomposition, $\k = \sum_{j = 0}^{2m-s} \gamma_{j} \phi_{s+j} + R,$ means that this result holds for surface splines themselves, and by Lemma \ref{GreensFns} it holds for polyharmonic kernels $G_m$ as well.
\begin{proof}
We begin by estimating the operator norm of $E_{\k}$. 
To do this for $1\le p\le \infty$, 
we simply find estimates for $\|E_{\k}\|_{1\to1}$ 
and $\|E_{\k}\|_{\infty\to \infty}$, 
obtaining the $\|E_{\k}\|_{p\to p}$ 
norm by interpolation.
By symmetry, both the $L_1$ and $L_{\infty}$ operator norms are bounded by 
$\sup_{\alpha \in \sphere} \int_{\sphere} |\e_{\k}(x, \alpha)|\dif x.$ 

The decomposition of the kernel 
permits us to estimate this integral as the sum of the constituent integrals $\int|\e_{\phi_{s+j}}(x, \alpha)|\dif x$, for $j=0 \dots {2m-s}$ and  
$\int|\e_{R}(x, \alpha)|\dif x$. 
The latter can be estimated using Lemma \ref{Joe_est_lemma} directly: 
$\int_x |\e_{R}(x, \alpha)|\dif x \le \frac{K}{2m!}\omega_d \rho^{2m} \|R^{(2m)}\|_{L_{\infty}[-1,1]}$ .
The integrals of the kernels $\e_{\phi_{s+j}}(x, \alpha) $ are estimated by splitting the sphere into the  southern hemisphere, $\Omega_1$, and northern hemisphere, $\Omega_1^{c}$ . 
By Lemma \ref{ss_loc},  $\e_{\phi_{s+j}}$  is bounded uniformly over 
$\Omega_1$  by $C \rho^{2m+1}$, so  
$\int_{ \Omega_1} |\e_{\phi_{s+j}}(x, \alpha)|\dif \alpha \le C \rho^{2m+1}$.

On $\Omega_1^{c}$ we integrate using polar coordinates, obtaining:
\begin{eqnarray*} 
\int_{\Omega_1^{c}} 
  |\e_{\phi_{s+j}}(x, \alpha)|\,
\dif \alpha%\\
&\le&
C
\int_{\Omega_1^{c}}
  \rho^{2(m+j)-d} \left(1+\frac{\dist(x,\alpha)}{\rho}\right)^{2j-d-1} 
\dif \alpha\\
&\le & 
C \rho^{2(m+j)}
\left(  1+ \int_{1}^{\pi/(2\rho)} R^{2j-2}\dif R\right)\\
&\le&   
C \rho^{2(m+j)} 
\left(1 + \left(\frac{\pi}{2} \right)^{2j-1} 
\rho^{1-2j}\right) 
\le 
C\rho^{2m}.
\end{eqnarray*}
 To bound the coefficients, we make the estimate 
\[
\sum_{\xi\in\Xi} |A_{\xi}|
\le 
\sum_{\xi\in\Xi} 
\int |\a(\alpha,\xi)| 
  |\L_{2m}(f- p_f)(\alpha)| \,
\dif \alpha .
\]
This is less than
\begin{eqnarray*}
\int 
\sum_{\xi\in\Xi} 
  |\a(\alpha,\xi)| |\L_{2m}(f- p_f)(\alpha)|\, 
\dif \alpha 
&\le&
K \int_{\sphere} |\L_{2m}(f- p_f)(\alpha)|\, \dif \alpha\\
&\le& C K \|f\|_{W_p^{2m}}.
\end{eqnarray*}
\end{proof}
The previous theorem requires the target function to have $2m$ derivatives in $L_p$, which is quite restrictive. To treat more general functions, we can first approximate a target function of lower smoothness by a nearby member, $g$, of $W_p^{2m}$, and apply
the theorem to $g$ instead of $f$. This is an old trick in approximation theory, and it is a consequence of the fact that the Besov spaces are interpolation spaces of Sobolev spaces. 
We make use of the Besov spaces $B_{p,\infty}^{\sigma}$, $1\le p<\infty$ and $0<\sigma < 2m$, which are the spaces of $L_p$ functions with norm  
\[\|f\|_{B_{p,\infty}^{\sigma}(\sphere)}:=\sup_{t>0}\left( t^{-\frac{\sigma}{2m}}\inf \left\{\|f-g\|_p + t \|g\|_{W_p^{2m}}: g\in W_p^{2m}(\sphere)\right\}\right).\]
When $p = \infty$, the norm can be rewritten with $C^{2m}$ replacing $W_p^{2m}$.
Rather than paraphrase the theory here, we point the interested reader to \cite[Chapters 1 and 7]{Tri} for the pertinent theorems and definitions.
%
%Cor: Low smoothness
%
\begin{corollary}
In the setting of the previous theorem, 
if $f\in B_{p,\infty}^{\sigma}(\sphere)$ for $1\le p\le\infty$ with $0<\sigma<2m$ 
then 
$\dist(f,S(\k,\Xi))_p 
\le 
\const \rho^{\sigma} 
\|f\|_{B_{p,\infty}^{\sigma}}$, 
and this can be accomplished with an approximant 
\[s_{\xi,f}(x) = \sum_{\xi\in \Xi} A_{\xi} \k(x\cdot \xi) + p(\xi),\]
with $p\in \Pi_{\mathcal{J}}$ and with coefficients satisfying
$\|A\|_{\ell_1(\Xi)} \le C \rho^{\sigma - 2m}\|f\|_{B_{p,\infty}^{\sigma}(\sphere)}.$
\end{corollary}
\begin{proof}
By real interpolation, we have, for every $t>0$, that 
$\inf 
\{\|f-g\|_p + t \|g\|_{W_p^{2m}}: 
g\in W_p^{2m}(\sphere)\} 
\le 
t^{\frac{\sigma}{2m}}
\|f\|_{B_{p,\infty}^{\sigma}(\sphere)}$. 
This implies, taking $t=\rho^{2m}$, that we can find 
$g_{\rho} \in W_p^{2m}(\sphere)$ satisfying 
\begin{eqnarray*} 
\|f-g_{\rho}\|_p &\le& 2 \rho^{\sigma}\|f\|_{B_{p,\infty}^{\sigma}(\sphere)};\label{Kfunctional_a}\\
\|g_{\rho}\|_{ W_p^{2m}(\sphere)} 
&\le& 2 \rho^{\sigma-2m}
\|f\|_{B_{p,\infty}^{\sigma}(\sphere)}\label{Kfunctional_b}
 \end{eqnarray*}
Applying the the previous theorem to $g_{\rho}$ gives
$\|f - s_{\Xi}g\|_p \le \|f-g\|_p + \|g-s_{\Xi}g\|_p \le 2 \rho^{\sigma}\|f\|_{B_{p,\infty}^{\sigma}(\sphere)}
 +\const \rho^{2m} \rho^{\sigma-2m}
\|f\|_{B_{p,\infty}^{\sigma}(\sphere)}.$
The coefficient estimate follows by a similar argument.
 \end{proof}

By Lemma \ref{dense}, we can  apply the previous results to approximation with sufficiently dense centers.
\begin{corollary}
For centers $\Xi \in \sphere$, having fill distance $h<h_0$, (with $h_0$ given by Lemma \ref{dense}), if $f\in X_p^s$  
\[\dist(f,S(\k,\Xi))_p\le \const (\k) h^{s} \|f\|_{X_p^{s}}\] 
where  $X_p^s$ is $W_p^{2m}$ when $s=2m$, or $B_{p,\infty}^s$ when $0<s<2m$. 
\end{corollary}

\trivlist
   \item[\hskip \labelsep{\hspace{\parindent}%
   \normalfont\itshape Acknowledgments.}]
The author is indebted to Joe Ward and Fran Narcowich for their substantial advice and many helpful discussions. 
He is also grateful for the many helpful comments from the referees.
{\endtrivlist}

%   \removelastskip\addvspace{18pt}\par\noindent
%   \normalfont\itshape\hspace*{20pt}%
%   \parbox[t]{180pt}%
%   {\raggedright\parindent -10pt
%    Thomas Hangelbroek\\
%    Department of Mathematics\\
%    Mailstop 3368\\
%    Texas A\&M University\\
%    College Station, TX 77843
%    U.S.A.
%    \texttt{hangelbr@math.tamu.edu}
% }\hspace*{23pt}%
%   \normalfont\upshape\ignorespaces
\end{document}